\documentclass[12pt]{amsart}
\usepackage{amsmath,amsthm,latexsym,amscd,amsbsy,amssymb,amsfonts,leqno,
euscript}

\input{xy}
\xyoption{all}

\numberwithin{equation}{section}

\newcommand{\LC[1]}{\mathrm{LC}^{#1}}

\newcommand{\Ddim}{\dim_\triangle}

\newcommand{\U}{\mathcal U}
\newcommand{\V}{\mathcal V}

\newcommand{\I}{\mathbb I}

\newcommand{\e}{\varepsilon}

\newtheorem{thm}{Theorem}[section]
\newtheorem{pro}[thm]{Proposition}
\newtheorem{lem}[thm]{Lemma}
\newtheorem{cor}[thm]{Corollary}

\begin{document}


\title[Spaces with fibered approximation property in dimension $n$]
{Spaces with fibered approximation property in dimension $n$}

\author{Taras Banakh}
\address{Uniwersytet Humanistyczno-Przyrodniczy Jana Kochanowskiego w Kielcach (Poland), and Ivan Franko National University of Lviv (Ukraine)}
\email{T.O.Banakh@gmail.com}

\author{Vesko  Valov}
\address{Department of Computer Science and Mathematics, Nipissing University,
100 College Drive, P.O. Box 5002, North Bay, ON, P1B 8L7, Canada}
\email{veskov@nipissingu.ca}
\thanks{The second author was partially supported by NSERC Grant 261914-08.}




\begin{abstract}
A metric space $M$ us said to have the fibered approximation
property in dimension $n$ (br., $M\in \mathrm{FAP}(n)$) if for any
$\epsilon>0$, $m\geq 0$ and any map $g\colon\I^m\times\I^n\to M$
there exists a map $g'\colon\I^m\times\I^n\to M$ such that $g'$ is
$\epsilon$-homotopic to $g$ and $\dim
g'\big(\{z\}\times\I^n\big)\leq n$ for all $z\in\I^m$. The class of
spaces having the $\mathrm{FAP}(n)$-property is investigated in this
paper. The main theorems are applied to obtain generalizations of
some results due to Uspenskij \cite{vu} and Tuncali-Valov
\cite{tv1}.
\end{abstract}

\maketitle

\markboth{}{Fibered approximation property}



\section{Introduction}
All spaces in the paper are assumed to be Tychonoff and all maps
continuous. By $C(X,M)$ we denote all maps from $X$ into $M$.

We say that a metric space $M$ has the fibered approximation
property in dimension $n$ (br., $M\in \mathrm{FAP}(n)$), where
$n\geq 0$, if for any $\epsilon>0$, any $m\geq 0$ and any map
$g\colon\I^m\times\I^n\to M$ there exists a map
$g'\colon\I^m\times\I^n\to M$ such that $g'$ is $\epsilon$-homotopic
to $g$ and $\dim g'(\{z\}\times\I^n)\leq n$ for all $z\in\I^m$.

In the paper we investigate the class of spaces having the
$\mathrm{FAP}(n)$-property, where $n\geq 0$. According to
\cite{tv1}, this class contains all Euclidean spaces. It is shown in
Theorem 2.7 below that a complete metric space has the
$\mathrm{FAP}(n)$-property if and only if it has locally the same
property. So, any Euclidean manifolds also has the
$\mathrm{FAP}(n)$-property, $n\geq 0$. Another
$\mathrm{FAP}(n)$-spaces are described in the last section. For
example, if $M$ is a manifold modeled on the $n$-dimensional Menger
cube, or $M=\I^n$, then $M\times Z$ has the
$\mathrm{FAP}(n)$-property for any completely metrizable space $Z$.

We also introduced a subclass of $\mathrm{FAP}(n)$-spaces, the {\em
strong $\mathrm{FAP}(n)$-spaces}, see Section 4. For example, any
product of finitely many 1-dimensional completely metrizable
$\mathrm{LC}(0)$-spaces without isolated points is a strong
$\mathrm{FAP}(n)$-space for all $n\geq 0$ (Corollary 4.5).

Next theorem is the main result in this paper.

\begin{thm}\label{ap}
Let $f\colon X\to Y$ be a perfect map with $\Ddim(f)\leq n$, where
$X$ and $Y$ are paracompact spaces. If $M\in \mathrm{FAP}(n)$ is
completely metrizable, then $\mathcal{R}_n^f(Y,M)=\{g\in C(X,M):\dim
g(f^{-1}(y))\leq n\hbox{~}\hbox{~}\mbox{for all}\hbox{~} y\in Y\}$
is a $G_\delta$-subset of $C(X,M)$ and every simplicially
factorizable map in $C(X,M)$ is homotopically approximated by maps
from $\mathcal{R}_n^f(Y,M)$.
\end{thm}

\begin{cor}
Let $f\colon X\to Y$ be a perfect $0$-dimensional surjection between
paracompact spaces and $M$ a completely metrizable $\mathrm{ANR}$.
Then the maps $g\in C(X,M)$ such that $\dim g(f^{-1}(y))=0$ for all
$y\in Y$ form a dense $G_\delta$-subset of $C(X,M)$.
\end{cor}

Corollary 1.2 was obtained in \cite{tv1} in the particular case when
$Y$ is a $C$-space and $M$ an Euclidean space (see also \cite{vu}
for the case $X$ compact, $Y$ a $C$-space and $M=\I$).

\begin{cor}
Let $M\in \mathrm{FAP}(n)$ be a completely metrizable $\mathrm{ANR}$
and $f\colon X\to Y$  a perfect $n$-dimensional surjection between
paracompact spaces with $Y$ being a $C$-space.   Then the maps $g\in
C(X,M)$ such that $\dim g(f^{-1}(y))\leq n$ for all $y\in Y$ form a
dense $G_\delta$-subset of $C(X,M)$.
\end{cor}

The version of Corollary 1.3 with $M$ being an Euclidean space was
established in \cite{tv1}.

Let us explain the notions in Theorem 1.1. A map $g\in C(X,M)$ is
homotopically approximated by maps from $\mathcal{H}$ means that for
every function $\e\in C(X,(0,1])$ there exists $g'\in\mathcal{H}$
which is $\e$-homotopic to $g$. Here, the maps $g$ and $g'$ are said
to be $\e$-homotopic, if there is a homotopy $h\colon X\times\I\to
M$ connecting $g$ and $g'$ such that each set $h(\{x\}\times\I)$ has
a diameter $<\e(x)$, $x\in X$.

The function space $C(X,M)$ appearing in this theorem is endowed
with the source limitation topology whose neighborhood base at a
given function $f\in C(X,M)$ consists of the sets
$$B_\rho(f,\e)=\{g\in C(X,M):\rho(g(x),f(x))<\e(x)\hbox{~}\forall\hbox{~}x\in X\},$$ where $\rho$ is a
fixed compatible metric on $M$ and $\e:X\to(0,1]$ runs over
continuous positive functions on $X$. If $X$ is paracompact, the
source limitation topology doesn't depend on the metric $\rho$ and
it has the Baire property provided $M$ is completely metrizable.

We say that a map $g\colon X\to M$ is simplicially factorizable
\cite{bv} if there exists a simplicial complex $L$ and two maps
$g_1\colon X\to L$ and $g_2\colon L\to M$ such that $g=g_2\circ
g_1$. In each of the following cases the set of simplicially
factorizable maps is dense in $C(X,M)$ (see \cite[Proposition
4]{bv}): (i) $M$ is an $\mathrm{ANR}$; (ii) $\dim X\leq k$ and $M$
is $\LC[k-1]$; (iii) $X$ is a $C$-space and $M$ is locally
contractible.

The dimension $\Ddim(f)$ was defined in \cite{bv}: $\Ddim(f)$ of a
map $f:X\to Y$ is equal to the smallest cardinal number $\tau$ with
the following property: for every open cover $\U$ of $X$ there is a
map $g:X\to \mathbb I^\tau$ such that the diagonal product $f\Delta
g:X\to Y\times \mathbb I^\tau$ is a $\U$-disjoint map. The last one
means that every $z\in (f\Delta g)(X)$ has a neighborhood $V$ such
that $(f\Delta g)^{-1}(V)$ is the union of a disjoint open in $X$
family refining $\U$. According to results from \cite{lev},
\cite{bp:98} and \cite{tv}, for any perfect map $f:X\to Y$ between
paracompact spaces we have: (i) $\dim f\leq \Ddim(f)$; (ii)
$\Ddim(f)=0$ iff $\dim f=0$; (iii) $\Ddim(f)=\dim f$ if $Y$ a
$C$-space; (iv) $\Ddim(f)\leq\dim f+1$ if the spaces $X,Y$ are
compact.


\section{Some properties of $\mathrm{FAP}(n)$-spaces}


Suppose that $(M,\rho)$ is  a  complete metric space and $Z\subset
M$ a closed set. If $f\colon X\to Y$ is a perfect surjective map
such that $X$ and $Y$ are paracompact and  $\dim f\leq n$, let



$$\mathcal{R}_n^f(H,Z)=\{g\in C(X,M):\dim g(f^{-1}(y))\cap Z\leq n\hbox{~}\forall y\in H\hbox{~}\}$$
with $H\subset Y$. Let also $\mathcal{R}_n^f(H,Z,k)$, where
$H\subset Y$ and $k\geq 1$, denote the set of all maps $g\in C(X,M)$
satisfying the following condition:

\begin{itemize}
\item Each set $\Gamma(g,y)=g(f^{-1}(y))\cap Z$, $y\in H$,
can be covered by an open family $\gamma(g,y)$ in $M$ of mesh $\leq
1/k$ and order $\leq n$.
\end{itemize}
Recall that the order of $\gamma(g,y)$ is $\leq n$ provided any
point of $M$ is contained in at most $n+1$ elements of
$\gamma(g,y)$.

\begin{lem}\label{intersection}
$\mathcal{R}_n^f(H,Z)$ is the intersection of all
$\mathcal{R}_n^f(H,Z,k)$, $k\geq 1$, for any $H\subset Y$.
\end{lem}

\begin{proof}
Let $g\in\mathcal R_n^f(H,Z)$. Then $\dim \Gamma(g,y)\leq n$ for all
$y\in H$. Hence, $\Gamma(g,y)$ admits an open in $M$ cover of mesh
$\leq 1/k$ and order $\leq n$ for any $k\geq 1$ and $y\in H$.
Therefore, $\mathcal R_n^f(H,Z)$ is contained in the intersection of
all $\mathcal{R}_n^f(H,Z,k)$, $k\geq 1$. On the other hand, if $g\in
C(X,M)$ belongs to this intersection and $y\in Y$ is fixed, then
each $\Gamma(g,y)$ admits open covers of arbitrary small mesh and
order $\leq n$. So, $\dim\Gamma(g,y)\leq n$ and $g\in\mathcal
R_n^f(H,Z)$.
\end{proof}

\begin{lem}\label{nbd}
Suppose $X$ and $Y$ are metric spaces and
$g\in\mathcal{R}_n^f(y,Z,k)$ for some $y\in Y$ and $k\geq 1$. Then
there exists a neighborhood $V_y$ of $y$ in $Y$ and $\delta_y>0$
such that $g'\in\mathcal{R}_n^f(y',Z,k)$ provided $y'\in V_y$ and
$g'\in C(X,M)$ with $\rho(g'(x),g(x))<\delta_y$ for all $x\in
f^{-1}(y')$. The same conclusion remains true if $Z=M$ and $X, Y$
paracompact.
\end{lem}

\begin{proof}
Assume first that $X$ and $Y$ are metric spaces. In case
$\Gamma(g,y)\neq\varnothing$, it can be covered by an open in $M$
family $\gamma(g,y)$ of mesh $\leq 1/k$ and order $\leq n$. Let
$G=\cup\gamma(g,y)$ and $\Pi=M\backslash G$. If
$\Gamma(g,y)=\varnothing$, let $\Pi=Z$. Hence, in both cases we have
$$Z\cap\Pi\cap g(f^{-1}(y))=\varnothing.\leqno{(1)}$$
It suffices to show there exists a neighborhood $V_y$ of $y$ in $Y$
such that
$$\delta_y=\rho\big(g(f^{-1}(V_y)),Z\cap\Pi\big)>0.$$  Indeed, otherwise
there would be a sequence $\{x_i\}_{i\geq 1}\subset X$ such that
$\{f(x_i)\}_{i\geq 1}$ converges to $y$ and
$\rho\big(g(x_i),Z\cap\Pi\big)\leq 1/i$, $i\geq 1$. Passing to a
subsequence, we may assume that $\{x_i\}_{i\geq 1}$ also converges
to a point $x\in f^{-1}(y)$. So, $g(x)\in Z\cap\Pi\cap
g(f^{-1}(y))$, which contradicts $(1)$.

If $Z=M$, we let $G=\cup\gamma(g,y)$ and
$\displaystyle\delta_y=\frac{1}{2}\rho\big(g(f^{-1}(y)),M\backslash
G\big)$, where $\gamma(g,y)$ is as above. Using that $f$ is perfect,
we can find a neighborhood $V_y$ of $y$ in $Y$ such that
$\rho\big(g(f^{-1}(V_y)),M\backslash G\big)\geq\delta_y$. Then $V_y$
and $\delta_y$ are as required.
\end{proof}

\begin{lem}\label{open}
Let $H\subset Y$ be closed. Then every $\mathcal{R}_n^f(H,Z,k)$ is
open in $C(X,M)$ in each of the following two cases: $(i)$ $Z\subset
M$ is closed and both $X$ and $Y$ are metric spaces; $(ii)$ $Z=M$
and $X,Y$ are paracompact.
\end{lem}

\begin{proof}
The lemma follows from the proof of \cite[Proposition 3.3]{bv1}. For
completeness, we provide the arguments. We consider only the first
case, the second one is similar. Suppose
$g_0\in\mathcal{R}_n^f(H,Z,k)$. Then, by Lemma~\ref{nbd}, for every
$y\in H$ there exist a neighborhood $V_y$ and a positive
$\delta_y\leq 1$ such that $g\in\mathcal{R}_n^f(y',Z,k)$ for any
$y'\in V_y$ provided $g|f^{-1}(y')$ is $\delta_y$-close to
$g_0|f^{-1}(y')$. The family $\{V_y\cap H:y\in H\}$ can be supposed
to be locally finite in $H$. Then the set-valued map $\varphi\colon
H\to (0,1]$, $\varphi(y)=\cup\{(0,\delta_z]:y\in V_z\}$ is lower
semi-continuous. By \cite[Theorem 6.2, p.116]{rs}, $\varphi$ admits
a continuous selection $\beta\colon H\to (0,1]$. Let
$\overline{\beta}:Y\to (0,1]$ be a continuous extension of $\beta$
and $\alpha=\overline{\beta}\circ f$. It suffices to show that if
$g\in C(X,M)$ with $\rho\big(g_0(x),g(x)\big)<\alpha(x)$ for all
$x\in X$, then $g\in\mathcal{R}_n^f(y,Z,k)$ for every $y\in H$. So,
we take such a $g$ and fix $y\in H$. Then there exists $z\in H$ with
$y\in V_{z}$ and $\alpha(x)\leq\delta_{z}$ for all $x\in f^{-1}(y)$.
Hence, $\rho\big(g_0(x),g(x)\big)<\delta_z$, $x\in f^{-1}(y)$.
Therefore, according to the choice of $V_z$ and $\delta_z$,
$g\in\mathcal{R}_n^f(y,Z,k)$.
\end{proof}

Lemmas 2.1 and 2.3 imply the following proposition.

\begin{pro}\label{gdelta}
Let $H\subset Y$ be a closed set. Then $\mathcal R_n^f(H,Z)$ is a
$G_\delta$-subset of $C(X,M)$ in any of the cases $(i)$ and $(ii)$
from Lemma $2.3$.
\end{pro}

Next lemma is very useful when dealing with homotopically dense
subsets of function spaces. Here, a set $U\subset C(X,M)$ is said to
be homotopically dense in $C(X,M)$ if for every $g\in C(X,M)$ and
$\e\in C(X,(0,1])$ there exists $g'\in U$ which is $\e$-homotopic to
$g$.

\begin{lem}\cite[Lemma 2.2]{bv}
Let $X$ be a metric space and $G\subset C(X,M)$. Suppose
$\{U(i)\}_{i\geq 1}$ is a sequence of open subsets of $C(X,M)$ such
that
\begin{itemize}
\item for any $h\in G$, $i\geq 1$ and any function $\eta\in
C(X,(0,1])$ there exists $g_i\in B_\rho(h,\eta)\cap U(i)\cap G$
which is $\eta$-homotopic to $h$.
\end{itemize}
Then, for any $g\in G$ and $\e\colon X\to (0,1]$ there exists
$g'\in\bigcap_{i=1}^{\infty} U(i)$ and an $\e$-homotopy connecting
$g$ and $g'$. Moreover, $g'|A=g_0|A$ for some $g_0\in C(X,M)$ and
$A\subset X$ provided $g_i|A=g_0|A$ for all $i$.
\end{lem}

\begin{cor}
Let $X$ be a metric space and $\{G_i\}_{i\geq 1}$ a sequence of
homotopically dense $G_\delta$-subsets of $C(X,M)$. Then the set
$\bigcap_{i=1}^{\infty}G_i$ is also homotopically dense in $C(X,M)$.
\end{cor}

\begin{proof}
Each $G_i$ is the intersection of a sequence $\{G_{ij}\}_{j\geq 1}$
of open sets in $C(X,M)$. Since $G_i$ is homotopically dense in
$C(X,M)$, so are all $G_{ij}$, $j\geq 1$. Then we apply Lemma 2.5
for the sequence $\{G_{ij}\}_{i,j\geq 1}$ with $G$ being the whole
space $C(X,M)$.
\end{proof}

We are going to show the local nature of the
$\mathrm{FAP}(n)$-properties.

\begin{thm}\label{local}
A complete metric space $M$ possesses the $\mathrm{FAP}(n)$-property
if and only if every $z\in M$ has a neighborhood $U_z\in
\mathrm{FAP}(n)$.
\end{thm}

\begin{proof}
It is easily seen that if $M\in \mathrm{FAP}(n)$, then every open
set $U\subset M$ also has the $\mathrm{FAP}(n)$-property. Suppose
every $z\in M$ has an open neighborhood $U_z\in \mathrm{FAP}(n)$.
Fix an integer $m\geq 0$ and consider the projection $\pi\colon
\I^m\times\I^n\to\I^m$. We need to prove that the set $\mathcal
R_n^\pi(\I^m,M)$ is homotopically dense in $C(\I^m\times\I^n,M)$. To
this end, using an idea from the proof of \cite[Theorem 3.6]{mv},
for every $z\in M$ choose a positive $\epsilon_z$ such that $U_z$
contains the closed ball $\overline{B(z,3\epsilon_z)}$ with center
$z$ and radius $3\epsilon_z$. Following the notations from the
beginning of this section (with $X$ replaced by $\I^m\times\I^n$ and
$Y$ by $\I^m$), we consider the sets $\mathcal R(z)=\mathcal
R_n^\pi(\I^m,\overline{B(z,\epsilon_z)})$, $z\in M$.

\textit{Claim $1$. Every $\mathcal{R}(z)$, $z\in M$, is a
homotopically dense $G_\delta$-subset of $C(\I^m\times\I^n,M)$.}

All $\mathcal{R}(z)$ are $G_\delta$-subsets of $C(\I^m\times\I^n,M)$
by Proposition 2.4. To show their homotopical density in
$C(\I^m\times\I^n,M)$, fix $z_0\in M$, $g_0\in C(\I^m\times\I^n,M)$
and $\epsilon>0$ with $\epsilon<\epsilon_{z_0}$. Let
$A_{z_0}=g_0^{-1}\big(\overline{B(z_0,2\epsilon_{z_0})}\big)$ and
$W_{z_0}=g_0^{-1}\big(B(z_0,3\epsilon_{z_0})\big)$. Choose finitely
many sets $K_i=A_i\times B_i$, $i=1,2,..,k$, such that
$A_i\subset\I^m$ and $B_i\subset\I^n$ are homeomorphic to $\I^m$ and
$\I^n$, respectively, and $A_{z_0}\subset
K=\bigcup_{i=1}^{i=k}K_i\subset W_{z_0}$. We can also suppose that
there exists a polyhedron $L$ such that $A_{z_0}\subset L\subset K$.
For every $i$ consider the set $$\mathcal R_i=\{h\in
C(K_i,U_{z_0}):\dim h(\{y\}\times B_i)\leq n\hbox{~}\forall y\in
A_i\hbox{~}\}\leqno{(2)}$$ and let $p_i\colon C(K,U_{z_0})\to
C(K_i,U_{z_0})$ be the restriction map $g\rightarrow g|K_i$, $g\in
C(K,U_{z_0})$. Obviously, $p_i$ are continuous. By Proposition 2.4,
each $\mathcal R_i$ is a $G_\delta$-subset of $C(K_i,U_{z_0})$.
Hence, all $p_i^{-1}(\mathcal R_i)$  are $G_\delta$-subsets of
$C(K,U_{z_0})$. Moreover, each $\mathcal R_i$ is homotopically dense
in $C(K_i,U_{z_0})$  because $U_{z_0}\in \mathrm{FAP}(n)$. This,
according to the Homotopy Extension Theorem, implies that
$p_i^{-1}(\mathcal R_i)$ are also homotopically dense in
$C(K,U_{z_0})$. So, by Corollary 2.6, $\mathcal
H=\bigcap_{i=1}^{i=p}p_i^{-1}(\mathcal R_i)$ is homotopically dense
in $C(K,U_{z_0})$. Then there exists a map $h\in\mathcal H$ which is
$\epsilon$-homotopic to $g_0|K$. Applying again the Homotopy
Extension Theorem for the maps $h|L$ and $g_0$, we obtain a map
$g^*\in C(\I^m\times\I^n,M)$ such that $g^*|L=h|L$ and $g^*$ is
$\epsilon$-homotopic to $g_0$. Let us show that
$g^*\in\mathcal{R}(z_0)$, or equivalently, $\dim
g^*(\{y\}\times\I^n)\cap\overline{B(z_0,\epsilon_{z_0})}\leq n$ for
every $y\in\I^m$. It is easily seen that $(g^*)^{-1}(z)\subset
A_{z_0}$ for every $z\in\overline{B(z_0,\epsilon_{z_0})}$. The last
inclusion yields that
$g^*(\{y\}\times\I^n)\cap\overline{B(z_0,\epsilon_{z_0})}\subset
h((\{y\}\times\I^n)\cap A_{z_0})$ for any $y\in\pi(A_{z_0})$ and
$g^*(\{y\}\times\I^n)\cap\overline{B(z_0,\epsilon_{z_0})}=\varnothing$
if $y\not\in\pi(A_{z_0})$. Therefore, the proof of the claim is
reduced to show that $\dim h((\{y\}\times\I^n)\cap A_{z_0})\leq n$
for any $y\in\pi(A_{z_0})$. And this is really true. Indeed, for any
such $y$ let $\Lambda(y)=\{i\leq k:y\in A_i\}$. Then
$(\{y\}\times\I^n)\cap A_{z_0}=\bigcup_{i\in\Lambda(y)}(\{y\}\times
B_i)\cap A_{z_0}$. Since $h|K_i\in\mathcal R_i$, by $(2)$ we have
$\dim h((\{y\}\times B_i)\cap A_{z_0})\leq n$ for every
$i\in\Lambda(y)$. Hence, $h((\{y\}\times\I^n)\cap A_{z_0})$ is the
union of its closed sets $h((\{y\}\times B_i)\cap A_{z_0})$,
$i\in\Lambda(y)$, each of dimension $\leq n$. So, $\dim
h((\{y\}\times\I^n)\cap A_{z_0})\leq n$ which completes the proof of
the claim.

Now, we can show that $\mathcal R_n^\pi(\I^m,M)$ is homotopically
dense in $C(\I^m\times\I^n,M)$. To this end, fix $g\in C(X,M)$ and
$\eta>0$, and choose finitely many points $z_i\in M$, $i=1,..,q$,
such that
$g(\I^m\times\I^n)\subset\bigcup_{i=1}^{i=q}B(z_i,\epsilon_{z_i}/2)$.
Let $\delta=\min\{\eta,\epsilon_{z_i}/2:i\leq q\}$. By the above
claim,  each $\mathcal{R}(z_i)$ is a homotopically dense
$G_\delta$-subset of $C(\I^m\times\I^n,M)$. Therefore, so is the set
$\bigcap_{i\leq q}\mathcal{R}(z_i)$ according to Corollary 2.6.
Hence, there exists $g'\in\bigcap_{i\leq q}\mathcal{R}(z_i)$ which
is $\delta$-homotopic to $g$. It is easily seen that
$g'(\I^m\times\I^n)\subset\bigcup_{i=1}^{i=q}B(z_i,\epsilon_{z_i})$,
so $g'(\{y\}\times\I^n)\subset\bigcup_{i\leq
q}g'(\{y\}\times\I^n)\cap\overline{B(z_i,\epsilon_{z_i})}$ for any
$y\in\I^m$. Observe that each set
$g'(\{y\}\times\I^n)\cap\overline{B(z_i,\epsilon_{z_i})}$, $i\leq
q$, is of dimension $\leq n$ because $g'\in\mathcal{R}(z_i)$. Hence,
$\dim g'(\{y\}\times\I^n)\leq n$ for all $y\in\I^m$. Thus,
$g'\in\mathcal R_n^\pi(\I^m,M)$. This completes the proof.
\end{proof}

Next proposition shows that in the definition of
$\mathrm{FAP}(n)$-spaces we can consider any product
$\I^m\times\I^k$, $m\geq 0$ and $k\leq n$.

\begin{pro}
If a metrizable space $M$ has the $\mathrm{FAP}(n)$-property, then
any map $g\colon\I^m\times\I^k\to M$, where $m\geq 0$ and $k\leq n$,
can be approximated by a map $g'\colon\I^m\times\I^k\to M$ such that
$\dim g'(\{z\}\times\I^k)\leq n$ for all $z\in\I^m$.
\end{pro}

\begin{proof}
Suppose $M$ has the $\mathrm{FAP}(n)$-property. Let $\epsilon>0$ and
$g\colon\I^m\times\I^k\to M$ with $k\leq m$. Take a retraction
$r\colon\I^n\to\I^k$ and consider the maps
$\pi_1\colon\I^m\times\I^n\to\I^m\times\I^k$ and
$\pi_2\colon\I^m\times\I^k\to\I^m$ defined, respectively, by
$\pi_1((z,x))=(z,r(x))$ and $\pi_2(z,y)=z$. Then
$\pi=\pi_2\circ\pi_1\colon\I^m\times\I^n\to\I^m$ is the natural
projection. Since $M\in \mathrm{FAP}(n)$, there exists $h\in
C(\I^m\times\I^n, M)$ which is $\epsilon$-homotopic to the map
$g\circ\pi_1$ and $\dim h(\{z\}\times\I^n)\leq n$ for all
$z\in\I^m$. Consequently, the map $g'=h|(\I^m\times\I^k)$ is
$\epsilon$-homotopic to $g$ and $\dim g'(\{z\}\times\I^k)\leq n$,
$z\in\I^m$.
\end{proof}

Next theorem provides a characterization of $\mathrm{FAP}(n)$-spaces
in terms of simplicial maps.

\begin{thm}
For a complete metric space $M$ the following conditions are
equivalent:
\begin{enumerate}
\item[(i)] $M$ possesses the $\mathrm{FAP}(n)$-property;
\item[(ii)] If $p\colon K\to L$ is at most $n$-dimensional simplicial map
 between finite simplicial complexes, then the set
$\mathcal R_n^p(L,M)$ is homotopically dense in $C(K,M)$;
\end{enumerate}
\end{thm}

\begin{proof}
$(i)\Rightarrow (ii)$ Suppose $M\in \mathrm{FAP}(n)$ and $p\colon
K\to L$ a simplicial map between finite simplicial complexes with
$\dim p=k\leq n$. Let $K^{(0)}$, $L^{(0)}$ be the set of vertices of
$K$ and $L$, respectively, and fix $g_0\in C(K,M)$ and $\epsilon>0$.

First, we assume that $K$ is a simplex. Then $L$ is also a simplex
 and, since $\dim p=k$, $p^{-1}(z)\cap K^{(0)}$ contains at most $k+1$
points for every vertex $z\in L^{(0)}$. Consequently, we can find a
map $e^{(0)}\colon K^{(0)}\to\sigma_k^{(0)}$ which is injective on
each set $p^{-1}(z)\cap K^{(0)}$, $z\in L^{(0)}$. Here, $\sigma_k$
is a $k$-dimensional simplex. This map induces an affine map
$e\colon K \to\sigma_k$. Then the diagonal map $h=p\triangle e\colon
K\to L\times\sigma_k$ is an affine embedding. So, there exists a
retraction $r\colon L\times\sigma_k\to K$ such that $h\circ r$ is
the identity on $h(K)$. Consider the projection $\pi\colon \colon
L\times\sigma_k\to L$. By Proposition 2.8, there exists a map
$\overline{g}\colon L\times\sigma_k\to M$ $\epsilon$-homotopic to
$g_0\circ r$ such that $\dim\overline{g}(\{z\}\times\sigma_k)\leq n$
for every $z\in L$. Then for the map $g'=\overline{g}\circ h$ we
have $\dim g'(p^{-1}(z))\leq n$ because $h(p^{-1}(z))$ is
homeomorphic to a subset of $\{z\}\times\sigma_k$. Moreover, it
follows that $g'$ is $\epsilon$-homotopic to $g_0$. Therefore, in
this case $\mathcal R_n^p(L,M)$ is homotopically dense in $C(K,M)$.

Now, we can prove the general case. Let $\{K_i:i\leq s\}$ be all
simplexes of $K$ and for each $i\leq s$ we denote $$\mathcal
H_i=\{g\in C(K,M):\dim g(p^{-1}(z)\cap K_i)\leq
n\hbox{~}\forall\hbox{~}z\in p(K_i)\}.$$ According to Proposition
2.4, $\mathcal H_i$ are $G_\delta$ in $C(K,M)$. It is easily seen
that $\mathcal R_n^p(L,M)=\bigcap_{i=1}^{i=s}\mathcal H_i$. So, by
Corollary 2.6, it suffices to show that each $\mathcal H_i$ is
homotopically dense in $C(K,M)$. Using the previous case, each set
$\mathcal K_i=\{g\in C(K_i,M):\dim g(p^{-1}(z)\cap K_i)\leq
n\hbox{~}\forall\hbox{~}z\in p(K_i)\}$ is homotopically dense in
$C(K_i,M)$. Therefore, there exists a map $g_i\in\mathcal K_i$ which
is $\epsilon$-homotopic to $g_0|K_i$. Then, by the Homotopy
Extension Theorem, $g_i$ can be extended to a map $\overline{g}_i\in
C(K,M)$ $\epsilon$-homotopic to $g_0$. Obviously,
$\overline{g}_i\in\mathcal H_i$. So, each $\mathcal H_i$ is
homotopically dense in $C(K,M)$ which completes the proof.

$(ii)\Rightarrow (i)$ This implication is trivial because any
projection $\pi\colon\I^m\times\I^n\to\I^m$ is a simplicial map with
respect to suitable triangulations of $\I^m$ and $\I^m\times\I^n$.
\end{proof}

\section{Proof of Theorem 1.1 and Corollaries 1.2 - 1.3}

In this section, following the notations from Section 2,  we assume
that $(M,\rho)$ is a completely metrizable $\mathrm{FAP}(n)$-space.
As we already observed, every $g\in C(X,M)$ is simplicially
factorizable provided $M$ is an $\mathrm{ANR}$. Moreover, if
$f\colon X\to Y$ is a perfect map between paracompact spaces, then
$\dim f=\Ddim (f)$ when either $\dim f=0$ or $Y$ is a $C$-space
\cite{bv}. Let us also note that every $\mathrm{ANR}$ has the
$\mathrm{FAP}(0)$-property.  Hence, the proofs of Corollaries 1.2
and 1.3 follow from Theorem 1.1.

By Proposition 2.4, $\mathcal R_n^f(Y,M)$ is a $G_\delta$-subset of
$C(X,M)$. So, to prove Theorem 1.1 it suffices to show that any
simplicially factorizable map in $C(X,M)$ can be approximated by
maps from $\mathcal R_n^f(Y,M)$. This will be done in Proposition
3.3 below.

Recall that a map $p\colon K\to L$ between two simplicial complexes
is a $PL$-map if $p(\sigma)$ is contained in a simplex of $L$ and
$p$ is linear on $\sigma$ for every simplex $\sigma\in K$.

\begin{lem}\label{compact}
Let $p\colon K\to\sigma$ be a $PL$-map between a finite simplicial
complex $K$ and a simplex $\sigma$ with $\dim p\leq n$. Suppose
$g_0\in C(K,M)$ such that $\dim g_0(p^{-1}(y))\leq n$ for all
$y\in\partial\sigma$, where $\partial\sigma$ is the boundary of
$\sigma$. Then, for every $\epsilon>0$ there exists a map
$g\in\mathcal R_n^p(\sigma,M)$ which is $\epsilon$-homotopic to
$g_0$ and $g|p^{-1}(\partial\sigma)=g_0|p^{-1}(\partial\sigma)$.
\end{lem}

\begin{proof}
We may assume that $p$ is simplicial  because any $PL$-map between
finite simplicial complexes is simplicial with respect to some
triangulations of the complexes. Let $\Omega=p^{-1}(\partial\sigma)$
and $G=\{g\in C(K,M):g|\Omega=g_0|\Omega\}$. All sets $U(k)=\mathcal
R_n^p(\sigma,M,k)$, $k\geq 1$, are open in $C(K,M)$ and their
intersection is $\mathcal R_n^p(\sigma,M)$, see Lemmas 2.1 and 2.3.
So, by Lemma 2.5, it suffices to show that each $U(k)$ has the
following property: any $g\in G$ can be homotopically approximated
by maps from $U(k)\cap G$.

So, fix $g\in G$, $k\geq 1$ and $\delta>0$. We are going to find
$h\in U(k)\cap G$ which is $\delta$-homotopic to $g$. Since
$g|\Omega=g_0|\Omega$, $g\in\mathcal R_n^p(y,M,k)$ for every
$y\in\partial\sigma$. Consequently, each $y\in\partial\sigma$ has a
neighborhood $V_y$ in $\sigma$ with corresponding $\delta_y>0$ both
satisfying the hypotheses of Lemma~\ref{nbd}. Choose finitely many
$y_i\in\partial\sigma$, $i\leq s$, such that $\displaystyle
V=\bigcup_{i\leq s}V_{y_i}$ covers $\partial\sigma$. Let
$F=\sigma\backslash V$ and
$\displaystyle\eta=\min\{\delta,\delta_{y_i}:i\leq s\}$.  We
consider such a triangulation $T$ of $\sigma$ that the complex
$L=\{\tau\in T:\tau\cap F\neq\varnothing\}$ is disjoint with
$\partial\sigma$. Because $K$ and $\sigma$ are finite complexes,
both they have triangulations $T_K$ and $T_\sigma$ such that
$T_\sigma$ is a subdivision of $T$ and $p$ remains simplicial with
respect to $T_K$ and $T_\sigma$. So, we can apply Theorem 2.9 to
find a map $g_1\in \mathcal R_n^p(\sigma,M)$ which is
$\eta$-homotopic to $g$. Then the map $g_2\colon \Omega\cup
p^{-1}(L)\to M$, $g_2|\Omega=g|\Omega$ and
$g_2|p^{-1}(L)=g_1|p^{-1}(L)$, is $\eta$-homotopic to $g|\Omega\cup
p^{-1}(L)$. Since $\Omega\cup p^{-1}(L)$ is a subcomplex of $K$, by
the Homotopy Extension Theorem, $g_2$ can be extended to a map
$h\colon K\to M$ which is $\eta$-homotopic to $g$. We have
$h\in\mathcal R_n^p(y,M,k)$ for all $y\in\sigma$. Indeed, this
follows from the choice of $V_{y_i}$ and $\delta_i$, $i\leq s$ (when
$y\in V$), and from  $g_1\in\mathcal R_n^p(\sigma,M)$ (when $y\in
L$). Hence, $h\in U(k)\cap G$ which completes the proof.
\end{proof}

Next step is to prove that the set $\mathcal R_n^f(L,M)$ is
homotopically dense in $C(N,M)$ for any perfect  $PL$-map $f\colon
N\to L$ between simplicial complexes with $\dim f\leq n$ .

\begin{lem}\label{simcomplex}
Let $N,L$ be simplicial complexes and $f\colon N\to L$ a perfect
$PL$-map with $\dim f\leq n$. Then $\mathcal R_n^f(L,M)$ is a
homotopically dense subset of $C(N,M)$
\end{lem}

\begin{proof}
We follow the arguments from the proof of \cite[Lemma 11.3]{bv}. Fix
$g\in C(N,M)$ and $\alpha\in C(N,(0,1])$. We are going to find
$h\in\mathcal R_n^f(L,M)$ which is $\alpha$-homotopic to $g$. Let
$L^{(i)}$, $i\geq 0$, be the $i$-dimensional skeleton of $L$ and put
$L^{(-1)}=\varnothing$ and $h_{-1}=g$. Construct inductively a
sequence $(h_i:N\to M)_{i\geq 0}$ of maps such that
\begin{itemize}
\item $h_{i}|f^{-1}(L^{(i-1)})=h_{i-1}|f^{-1}(L^{(i-1)})$;
\item $\displaystyle h_{i}$ is $\displaystyle\frac{\alpha}{2^{i+2}}$-homotopic to $h_{i-1}$;
\item $\dim h_i(f^{-1}(y))\leq n$ for every $y\in L^{(i)}$.
\end{itemize}

Assuming that the map $h_{i-1}:N\to M$ has been constructed,
consider the complement $L^{(i)}\setminus L^{(i-1)}=\sqcup_{j\in
J_i}\overset{\circ}\sigma_j$, which is the discrete union of open
$i$-dimensional simplexes.  Since, by \cite[Lemma 4.1]{bv}, each
$f^{-1}(\sigma_j)$ is a finite subcomplex of $N$, and $\dim
h_{i-1}(f^{-1}(y))\leq n$ for every $y\in L^{(i-1)}$, we can apply
Lemma~\ref{compact}  to find a map $g_j:f^{-1}(\sigma_j)\to M$,
$j\in J_i$, such that
\begin{itemize}
\item $g_j$ coincides with $h_{i-1}$ on the set $f^{-1}(\sigma^{(i-1)}_j)$;
\item $g_j$ is $\displaystyle\frac{\alpha}{2^{i+2}}$-homotopic to
$h_{i-1}$;
\item  $\dim g_j(f^{-1}(y))\leq n$ for every $y\in\sigma_j$.
\end{itemize}

Define a map $\varphi_i:f^{-1}(L^{(i)})\to M$ by the formula
$$\varphi_i(x)=\begin{cases} h_{i-1}(x)&\mbox{if $x\in
f^{-1}(L^{(i-1)})$;}\\ g_j(x)&\mbox{if $x\in f^{-1}(\sigma_j)$.}
\end{cases}$$ It can be shown that $\varphi_i$ is
$\displaystyle\frac{\alpha}{2^{i+2}}$-homotopic to
$h_{i-1}|f^{-1}(L^{( (i)})$. Moreover, $f^{-1}(L^{(i)})$ is a
subcomplex of $N$ (according to  \cite[Lemma 4.1]{bv}). So, by the
Homotopy Extension Theorem, there exists a continuous extension
$h_i:N\to M$ of the map $\varphi_i$ which is
$\displaystyle\frac{\alpha}{2^{i+2}}$-homotopic to $h_{i-1}$. The
map $h_i$ satisfies the inductive conditions.

Then the limit map $h=\lim_{i\to\infty}h_i:N\to M$ is well-defined,
continuous and $\alpha$-homotopic to $g$. Finally, since
$h|f^{-1}(L^{(i)})=h_i|f^{-1}(L^{(i)})$ for every $i\geq 0$,
$h\in\mathcal R_n^f(L,M)$.
\end{proof}

Now, we can complete the proof of Theorem 1.1.

\begin{pro}\label{general reg}
Let $f\colon X\to Y$ be a perfect map between paracompact spaces
with $\Ddim(f)\leq n$. Then every simplicially factorizable map
$g\in C(X,M)$ can be homotopically approximated by simplicially
factorizable maps $h\in C(X,M)$ such that $\dim h(f^{-1}(y))\leq n$
for every $y\in Y$.
\end{pro}

\begin{proof}
We follow the construction from the proof of \cite[Proposition
3.4]{bv1}. Fix a simplicially factorizable map $g\in C(X,M)$ and
$\epsilon\in C(X,(0,1])$. Then there exist a simplicial complex $D$
and maps $g_D\colon X\to D$, $g^D\colon D\to M$ with $g=g^D\circ
g_D$. The metric $\rho$ induces a continuous pseudometric $\rho_D$
on $D$, $\rho_D(x,y)=\rho(g^D(x),g^D(y))$. Since $D$ is a
neighborhood retract of a locally convex space (see \cite{ca} and
\cite{si}) and any sufficiently close maps from a given space into
$D$ are homotopic, we apply \cite[Lemma 8.1]{bv} to find an open
cover $\U$ of $X$ satisfying the following condition: if
$\alpha\colon X\to K$ is a $\U$-map into a paracompact space $K$
(i.e., $\alpha^{-1}(\omega)$ refines $\U$ for some open cover
$\omega$ of $K$), then there exists a map $q'\colon G\to D$, where
$G$ is an open neighborhood of $\overline{\alpha(X)}$ in $K$, such
that $g_D$ and $q'\circ\alpha$ are $\epsilon/2$-homotopic with
respect to the pseudometric $\rho_D$. Let $\U_1$ be an open cover of
$X$ refining $\U$ with $\inf\{\epsilon(x):x\in U\}>0$ for all
$U\in\U_1$.

Next, according to \cite[Theorem 6]{bv}, there exists a locally
finite open cover $\V$ of $Y$ such that: for any $\V$-map
$\beta\colon Y\to L$ into a simplicial complex $L$ we can find an
$\U$-map $\alpha\colon X\to K$ into a simplicial complex $K$ and a
perfect $PL$-map $p\colon K\to L$ with $\beta\circ f=p\circ\alpha$
and $\dim p\leq\Ddim f$.  Take $L$ to be the nerve of the cover $\V$
and $\beta\colon Y\to L$ the corresponding natural map. Then there
are a simplicial complex $K$ and maps $p$ and $\alpha$ satisfying
the above conditions. Hence, the following diagram is commutative:
$$\xymatrix{
X\ar[d]_f\ar[r]^\alpha& K\ar[d]^p\\
Y\ar[r]_\beta&L
}
$$
The choice of the cover $\U$ guarantees the existence of a map
$\varphi_D\colon G\to D$, where $G\subset K$ is an open neighborhood
of $\overline{\alpha(X)}$, such that $g_D$ and
$h_D=\varphi_D\circ\alpha$ are $\epsilon/2$-homotopic with respect
to $\rho_D$. Then, according to the definition of $\rho_D$,
$h'=g^D\circ \varphi_D\circ\alpha$ is $\epsilon/2$-homotopic to $g$
with respect to $\rho$. Replacing the triangulation of $K$ by a
suitable subdivision, we may additionally assume that no simplex of
$K$ meets both $\overline{\alpha(X)}$ and $K\backslash G$. So, the
union $N$ of all simplexes $\sigma\in K$ with
$\sigma\cap\overline{\alpha(X)}\neq\varnothing$ is a subcomplex of
$K$ and $N\subset G$. Moreover, since $N$ is closed in $K$,
$p_N=p|N\colon N\to L$ is a perfect map and $\dim p_N\leq\Ddim f$.
Therefore, we have the following commutative diagram, where $N$ and
$L$ are finite complexes, $p_N$ is a $PL$-map and $\varphi=g^D\circ
\varphi_D$:
$$\xymatrix{
N\ar[d]_{p_N}&X\ar[d]^f\ar[l]_{\alpha}\ar[r]^{\varphi\circ\alpha}&M\\
L&Y\ar[l]^\beta\ar[ru]_{\varphi}
}$$

Using that $\alpha$ is a $\U_1$-map and $\inf\{\epsilon(x):x\in
U\}>0$ for all $U\in\U_1$, we can construct a continuous function
$\epsilon_1:N\to(0,1]$ with $\epsilon_1\circ\alpha\leq\epsilon$.
 Then, by Lemma~\ref{simcomplex}, there exists a map
$\varphi_1\in C(N,M)$ which is $\epsilon_1/2$-homotopic to $\varphi$
and $\dim\varphi_1(p_N^{-1}(z))\leq n$ for every $z\in L$. Let
$g'=\varphi_1\circ\alpha$. Obviously, $g'$ is simplicially
factorizable. It is easily seen that $g'$ and $g$ are
$\epsilon$-homotopic and $g'(f^{-1}(y))\subset
\varphi_1(p_N^{-1}(\beta(y)))$ for all $y\in Y$. So, $\dim
g'(f^{-1}(y))\leq\dim\varphi_1(p_N^{-1}(\beta(y)))\leq n$. The proof
is completed.
\end{proof}

\section{Some more examples of $\mathrm{FAP}(n)$-spaces}

The class of $AP(n,0)$-spaces was introduced by the authors in
\cite{bv1}: we say that a metrizable space $M$ has the {\em
$\mathrm{AP}(n,0)$-approximation property} (br.,
$M\in\mathrm{AP}(n,0))$ if for every $\e>0$ and a map
$g\colon\I^n\to M$ there exists a $0$-dimensional map
$g'\colon\I^n\to M$ which is $\e$-homotopic to $g$. Next proposition
provides a wide class of spaces with the $\mathrm{FAP}(n)$-property.

\begin{pro}
Let $M_1\in\mathrm{AP}(n,0)$ be a completely metrizable
$n$-dimensional space, $n\geq 0$. Then $M_1\times M_2$ has the
$\mathrm{FAP}(n)$-property for any completely metrizable space
$M_2$.
\end{pro}

\begin{proof}
We are going to show that every map
$g=(g_1,g_2)\colon\I^m\times\I^n\to M_1\times M_2$, where $m\geq 0$,
can be homotopically approximated by a map $h\in C(\I^m\times\I^n,
M_1\times M_2)$ with $\dim h(\{z\}\times\I^n)\leq n$ for all
$z\in\I^m$. Denote by $\pi$ the projection
$\pi\colon\I^m\times\I^n\to \I^m$. Since $M_1$ has the
$\mathrm{AP}(n,0)$-property, the map $g_1\colon\I^m\times\I^n\to
M_1$ can be homotopically approximated by a map
$h_1\colon\I^m\times\I^n\to M_1$ such that all restrictions
$h_1|(\{z\}\times\I^n)$, $z\in\I^m$, have $0$-dimensional fibers,
see \cite[Theorem 1.1]{bv1}. This means that the diagonal product
$\pi\triangle h_1\colon\I^m\times\I^n\to\I^m\times M_1$ is a
0-dimensional map. It follows from our definition that every
metrizable space has the $\mathrm{FAP}(0)$-property. So, by Theorem
1.1, the map $g_2\colon\I^m\times\I^n\to M_2$ can be homotopically
approximated by a map $h_2\colon\I^m\times\I^n\to M_2$ such that all
images $h_2\big((\{z\}\times\I^n)\cap h_1^{-1}(z_1)\big)$,
$(z,z_1)\in\I^m\times M_1$, are $0$-dimensional. Then
$h=(h_1,h_2)\colon\I^m\times\I^n\to M_1\times M_2$ approximates $g$.
For any $z\in\I^m$ consider the map $p_z:h(\{z\}\times\I^n)\to
h_1(\{z\}\times\I^n)$, $p_z(h(z,t))=h_1((z,t))$, $t\in\I^n$. Observe
that $\dim h_1(\{z\}\times\I^n)\leq n$ (recall that $\dim M_1=n$)
and $p_z^{-1}(z_1)=h_2\big((\{z\}\times\I^n)\cap h_1^{-1}(z_1)\big)$
for any $z_1\in h_1(\{z\}\times\I^n)$. So, $\dim p_z=0$. According
to the dimension-lowering Hurewicz  theorem, $\dim
h(\{z\}\times\I^n)\leq\dim h_1(\{z\}\times\I^n)+\dim p_z\leq n$.
This completes the proof.
\end{proof}

Since every space with the disjoint $(n-1)$-disks property
$\mathrm{DD^{n-1}P}$, in particular, every manifold modeled on the
$n$-dimensional Menger cube or the $n$-dimensional N\:{o}beling
space, is an $\mathrm{AP}(n,0)$-space, see \cite[Corollary
6.5]{bv1}, we have the following

\begin{cor}
Let $X$ be a completely metrizable $n$-dimensional space with the
disjoint $(n-1)$-disks property. Then $X\times M$ has the
$\mathrm{FAP}(n)$-property for any completely metrizable space $M$.
\end{cor}

Now, we introduced a subclass of the $\mathrm{FAP}(n)$-spaces: a
metric space $M$ is said to be a {\em strong
$\mathrm{FAP}(n)$-space} if $M\in\mathrm{FAP}(k)$ for all $k\leq n$.
This is equivalent to the following condition: any map $g\in
C(\I^m\times\I^k,M)$, where $m\geq 0$ and $k\leq n$, can be
homotopically approximated by a map $g'\in C(\I^m\times\I^k,M)$ with
$\dim g'(\{z\}\times\I^k)\leq k$ for all $z\in\I^m$.

The local nature of strong $\mathrm{FAP}(n)$-spaces follows from
Theorem 2.7.

\begin{thm}
A complete metric space $M$ has the strong
$\mathrm{FAP}(n)$-property iff every $z\in M$ has a neighborhood
with the same property.
\end{thm}

By \cite{tv1}, any Euclidean space possesses the strong
$\mathrm{FAP}(n)$-property for all $n\geq 0$. More general examples
of strong $\mathrm{FAP}(n)$-spaces are provided by next proposition.

\begin{pro}
Let each $M_i$, $i=1,2,..,n$, be a completely metrizable
$\mathrm{LC}^{0}$-space without isolated points. Then the product
$\prod_{i=1}^{i=n}M_i$ is a strong $\mathrm{FAP}(n)$-space.
\end{pro}

\begin{proof}
According to \cite[Corollary 6.3]{bv1}, any product of $k$ many
completely metrizable spaces $\mathrm{LC}^{0}$-space without
isolated points has the $\mathrm{AP}(k,0)$-property. Then
Proposition 4.1 completes the proof.
\end{proof}

\begin{cor}
Any product of finitely many completely metrizable $1$-dimensional
$\mathrm{LC}^{0}$-spaces without isolated points has the
$\mathrm{FAP}(n)$-property for all $n\geq 0$.
\end{cor}

\end{document}